\documentclass[10pt,a4paper,reqno]{amsart}
\usepackage{amssymb}
\usepackage{amsmath}
\usepackage{enumerate}
\usepackage{graphicx}
\usepackage{tjmm}
\usepackage{multirow}

\usepackage[utf8x]{inputenc}

\begin{document}

\title[Reduced-order modelling based on Koopman operator theory]{Reduced-order modelling based on Koopman operator theory}
\author{Diana A. Bistrian, Gabriel Dimitriu, Ionel M. Navon}
\address{
 (D.A. Bistrian) University Politehnica of Timisoara\newline
\indent Department of Electrical Engineering and Industrial Informatics\newline
\indent Revolutiei Nr.5, 331128 Hunedoara, Romania}
\email{Corresponding author: diana.bistrian@upt.ro}

\address{
 (G. Dimitriu) University of Medicine and Pharmacy “Grigore T. Popa”\newline
\indent Department of Mathematics and Informatics\newline
\indent Universitatii 16, 700115 Iasi, Romania}
\email{gabriel.dimitriu@umfiasi.ro}

\address{
 (I.M. Navon) Florida State University\newline
\indent Department of Scientific Computing\newline
\indent Dirac Science Library Building, Tallahassee, FL 32306-4120, USA}
\email{Inavon@fsu.edu}

\setcounter{page}{17}% to be added afterwards
\coordinates{15}{2023}{1-2}{17-27}% to be added afterwards

% AMS Subject Classification (2010)

\begin{abstract}

The present study focuses on a subject of significant interest in fluid dynamics: the identification of a model with decreased computational complexity from numerical code output using Koopman operator theory. A reduced-order modelling method that incorporates a novel strategy for identifying the most impactful Koopman modes was used to numerically approximate the Koopman composition operator.

\end{abstract}

\maketitle

\section{Introduction}

Despite the fact that complex nonlinear dynamical systems can appear challenging to understand, the existence of similar flow characteristics indicates that a number of different dynamic phenomena are likely governed by the same fundamental processes. Using the modal decomposition method \cite{Holmes1996, Kaiser2018, ChKutz2019} is an effective strategy to find a helpful low-dimensional reference frame for capturing prominent dynamical processes. Model order reduction approaches based on modal decomposition have made significant improvements in the recent decade \cite{Arcucci2023, Mou2023, Iliescu2022, Daba2023, Sanfilippo2023}.

The choice of an appropriate reduced order basis to describe the system dynamics in relation to the system order reduction is the main topic of this study. 
The Koopman operator theory \cite{Koopm1931} provides the mathematical foundation for determining the reduced-order model of a complex nonlinear dynamical system. 

There are various advantages to transposing the nonlinear dynamics into a reduced-order model, which is a linear model by construction. These include the ability to identify the dominant frequencies, the production of a mathematical reduced-order model with higher fidelity, and a notable increase in computing speed.

The authors have made a substantial contribution to the advancement of modal decomposition techniques \cite{BistrianDimitriuNavon2019, BistrianDimitriuNavon2020a, BistrianDimitriuNavon2020b} and the introduction of new numerical algorithms for the modeling of nonlinear dynamical systems with lower computational complexity \cite{BistrianNavon2017b, Bistrian2022a, Bistrian2022b}  in the past years.

The present work contains a presentation of the mathematical aspects of modal decomposition technique based on Koopman operator theory \cite{Koopm1931}, with application to the Saint-Venant nonlinear dynamical system model.

The remainder of the article is organized as follows.
Section 2 discusses the mathematical considerations on the Koopman operator theory.
Section 3 presents the numerical method developed for reduced-order modelling.
Section 4 presents the test problem, consisting in nonlinear Saint-Venant equations dynamical model.  A qualitative study of the reduced-order model is conducted in the case of two experiments.
A summary and conclusions are given in Section 5.

\section{Mathematical considerations on the Koopman operator}
 
Let $\Omega  \subset {\mathbb{R}^n}$ be a compact and non-empty space. 
Let
\begin{equation}
{L^2}\left( \Omega  \right) = \left\{ {\psi :\Omega  \to \mathbb{R}{{\left| {\;\int_\Omega  {\left| \psi  \right|} } \right.}^2}d\Omega  < \infty } \right\}
\end{equation}
be a Hilbert space of square integrable functions on  $\Omega $, endowed with the inner product $\left\langle {{\psi _i},{\psi _j}} \right\rangle  = \int_\Omega  {{\psi _i}} {\psi _j}d\Omega $ and the norm $\left\| \psi  \right\| = \sqrt {\left\langle {\psi ,\psi } \right\rangle } $ for $\psi  \in {L^2}\left( \Omega  \right)$.

Let us consider a nonautonomous continuous-time dynamical system on  domain $\Omega  \subset {\mathbb{R}^n}$ governed by the nonlinear ordinary differential equation
\begin{equation}\label{system1}
\left\{ \begin{array}{l}
\frac{{dy}}{{dt}}\left( {{\bf{x}},t} \right) = f\left( {y,u,t} \right),\quad t \in {\mathbb{R}_{ \ge 0}}\\
y\left( {{\bf{x}},{t_0}} \right) = {y_0}\left( {\bf{x}} \right)
\end{array} \right.
\end{equation}
where the map $f$ is locally Lipschitz continuous, ${\bf{x}} \in {\mathbb{R}^n}$ is the Cartesian coordinate vector, $u \in {\mathbb{R}^m}$ is the input vector, with $n \gg m$. Forward invariance of the set $\Omega  \subset {\mathbb{R}^n}$ w.r.t. system dynamics (\ref{system1}) is assumed, i.e. any solution $y\left( {\textbf{x},t} \right) \in \Omega ,\;t \ge 0$ holds for all ${y_0}$.

\begin{definition}
\textbf{Dimensionality reduction in reduced-order modelling.}
The principle of modal reduction aims to finding an approximation solution of the form
\begin{equation}\label{rom1}
\left\{ \begin{array}{l}
y\left( {{\bf{x}},t} \right) \approx \sum\limits_{j = 1}^p {{a_j}\left( t \right){\psi _j}\left( {\bf{x}} \right)} ,\quad t \in {\mathbb{R}_{ \ge 0}}\\
\frac{{d{a_j}\left( t \right)}}{{dt}} = g\left( t \right),\;{a_j}\left( {{t_0}} \right) = a_j^0
\end{array} \right.
\end{equation}
expecting that this approximation becomes exact as $p \to \infty $, assuring preservation of dynamic stability, computational stability, and a small global approximation error compared to the true solution of (\ref{system1}).
\end{definition}

Let us consider a scalar observable function $\varphi :\Omega  \to \mathbb{C}$, $u = \varphi \left( y \right)$, $y \in \Omega $, $t \in {\mathbb{R}_{ \ge 0}}$  with a smooth and Lipschitz continuous flow ${F^t}:\Omega  \to \Omega $:
\begin{equation}\label{flow}
{F^t}\left( {{y_0}} \right) = {y_0} + \int_{{t_0}}^{{t_0} + t} {f\left( {y\left( \tau  \right)} \right)d\tau },
\end{equation}
which is forward-complete, i.e. the flow ${F^t}\left( y \right)$ has a unique solution on ${\mathbb{R}_{ \ge 0}}$ from any initial condition ${y_0}$.

The system class that fits the aforementioned assumption is quite vast and encompasses a wide range of physical systems, including whirling flows, shallow water flows, convection-diffusion processes, and so on.

The Koopman operator describes the propagation of state space observables over time. An observable might be any sort of system measurement or the dynamical reaction of the system. The recurrence of a fixed time-t flow map, i.e. sequential compositions of the map with itself, is assumed to describe the dynamical development. 

\begin{definition}
\textbf{Koopman operator.}
For dynamical systems of type (\ref{system1}), the semigroup of Koopman operators ${\left\{ {{{\mathcal K}^t}} \right\}_{t \in {\mathbb{R}_{ \ge 0}}}}:\Omega  \to \Omega $ acts on scalar observable functions $\varphi :\Omega  \to \mathbb{C}$ by composition with the flow semigroup ${\left\{ {{F^t}} \right\}_{t \in {\mathbb{R}_{ \ge 0}}}}$ of the vector field $f$:
\begin{equation}\label{koop}
{{\mathcal K}^t}\varphi  = \varphi \left( {{F^t}} \right).
\end{equation}
The Koopman operator is also known as the composition operator.
\end{definition}

\begin{proposition}
\textbf{Linearity of the Koopman operator.}
Consider the Koopman operator ${{\mathcal K}^t}$ and two observables ${\varphi _1},{\varphi _2} \in \Omega $ and the scalar $\alpha  \in \mathbb{R}$. Using (\ref{koop}) it follows that:
\begin{equation}\label{linear}
{{\mathcal K}^t}\left( {\alpha {\varphi _1} + \beta {\varphi _2}} \right) = \left( {\alpha {\varphi _1} + \beta {\varphi _2}} \right)\left( {{F^t}} \right) = \alpha {\varphi _1}\left( {{F^t}} \right) + \beta {\varphi _2}\left( {{F^t}} \right) = \alpha {{\mathcal K}^t}{\varphi _1} + \beta {{\mathcal K}^t}{\varphi _2}.
\end{equation}
\end{proposition}

\begin{definition}
\textbf{Infinitesimal generator.}
Let us assume that there is a generator ${{\mathcal G}_{\mathcal K}}:{\mathcal F} \to \Omega $, ${\mathcal F}$ being the domain of the generator and $\Omega $ the Banach space of observables. The operator ${{\mathcal G}_{\mathcal K}}$ stands as the infinitesimal generator of the time-t indexed semigroup of Koopman operators ${\left\{ {{{\mathcal K}^t}} \right\}_{t \in {\mathbb{R}_{ \ge 0}}}}$, i.e.
\begin{equation}\label{gen}
{{\mathcal G}_{\mathcal K}}\varphi  = \mathop {\lim }\limits_{t \searrow 0} \frac{{{{\mathcal K}^t}\varphi  - \varphi }}{t} = \frac{{d\varphi }}{{dt}}.
\end{equation}
\end{definition}

\begin{definition}
\textbf{Koopman eigenfunction.}
An observable $\phi  \in \Omega $ is called a Koopman eigenfunction if it satisfies the relation:
\begin{equation}\label{koopeig}
{{\mathcal G}_{\mathcal K}}\phi \left( y \right) = \frac{{d\phi }}{{dt}}\left( y \right) = s\phi \left( y \right),
\end{equation}
associated with the complex eigenvalue $s \in \mathbb{C}$.
\end{definition}

\begin{definition}
\textbf{Koopman mode.}
Let ${\phi _i} \in \Omega $ be an eigenfunction for the Koopman operator, corresponding to eigenvalue ${\lambda _i}$. For an observable $\varphi :\Omega  \to \mathbb{C}$, the Koopman mode corresponding to ${\phi _i}$ is the projection of $\varphi $ onto $span\left\{ {{\phi _i}} \right\}$. 
\end{definition}

\begin{theorem}
\textbf{Koopman Spectral Decomposition.}
Any observable $\varphi :\Omega  \to \mathbb{C}$  admits a Koopman spectral decomposition of the following form:
\begin{equation}\label{ksd}
\varphi \left( y \right) = \sum\limits_{j = 1}^\infty  {{a_j}\left( \varphi  \right)\lambda _j^t} {\phi _j},
\end{equation}
where $\lambda _j^t = {e^{{s_j}t}}$ w.r.t. ${s_j} = {\sigma _j} + i{\omega _j}$ with eigen-decay/growth ${\sigma _j}$ and eigenfrequencies ${\omega _j}$.
\end{theorem}

\begin{proof}
The Koopman mode decomposition of form (\ref{ksd}), first provided in \cite{Mezic2005}, was later coupled with numerical approaches for modal decomposition, such as dynamic mode decomposition \cite{SchmidSesterhen2008, Schmid2010, Mezic2013, Tu2014}.
Since for any semigroup of Koopman operators ${\left\{ {{{\mathcal K}^t}} \right\}_{t \in {\mathbb{R}_{ \ge 0}}}}$, exists an infinitesimal generator ${{\mathcal G}_{\mathcal K}}$, the following relation is satisfied for any ${\lambda ^t} = {e^{st}}$:
\begin{equation}\label{relgen}
{{\mathcal K}^t}\phi \left( y \right) = \phi \left( {{F^t}\left( y \right)} \right) = {\lambda ^t}\phi \left( y \right).
\end{equation}

Let us consider that the space $\Omega $ is chosen to be a Banach algebra, i.e. the set of eigenfunctions forms an Abelian semigroup under product of functions. If ${\phi _1},{\phi _2} \in \Omega $ are two eigenfunctions of the composition operator ${{\mathcal K}^t}$ with eigenvalues ${\lambda _1},{\lambda _2}$, then the function product ${\phi _1}{\phi _2}$ is also an eigenfunction of ${{\mathcal K}^t}$ with the eigenvalue ${\lambda _1}{\lambda _2}$. Thus, products of eigenfunctions are, again, eigenfunctions. It follows that, for any  observable function written in the following form:
\begin{equation}
\varphi \left( y \right) = \sum\limits_{j = 0}^\infty  {{a_j}\left( \varphi  \right){\phi _j}\left( y \right)} ,
\end{equation}
the Koopman operator acts as follows:
\begin{equation}
{{\mathcal K}^t}\varphi  = \sum\limits_{j = 1}^\infty  {{a_j}\left( \varphi  \right)\left( {{{\mathcal K}^t}{\phi _j}} \right)}  = \sum\limits_{j = 1}^\infty  {{a_j}\left( \varphi  \right)\lambda _j^t} {\phi _j}.
\end{equation}
\end{proof}

\section{Reduced-order modelling based on Koopman operator}

Dynamic Mode Decomposition (DMD) \cite{SchmidSesterhen2008, Schmid2010,Tu2014, Bistrian2015} , is a data-driven approach for estimating the modes and eigenvalues of the Koopman operator without numerically executing a Laplace transform. DMD has emerged as a popular approach for finding spatial-temporal coherent patterns in high-dimensional data, with a strong connection to nonlinear dynamical systems via the Koopman mode theory \cite{Koopm1931, chenoptimal2012}. We present in the following an improved numerical algorithm based on dynamic mode decomposition.

Let us consider a set of observables in the following form:
\begin{equation}
{u_i}\left( {{\bf{x}},t} \right) = u\left( {{\bf{x}},{t_i}} \right),\quad {t_i} = i\Delta t,\quad i = 0,...,{N_t}
\end{equation}
at a constant sampling time $\Delta t$, ${\bf{x}}$ representing the spatial coordinates, whether Cartesian or Cylindrical. 

A data matrix whose columns represent the individual data samples, called the snapshot matrix, is constructed in the following manner:
\begin{equation}
V = \left[ {\begin{array}{*{20}{c}}
{{u_0}}&{{u_1}}&{...}&{{u_{{N_t}}}}
\end{array}} \right] \in {\mathbb{R}^{{N_x} \times ({N_t} + 1)}}
\end{equation}
Each column ${u_i}$ is a vector with ${N_x}$ components, representing the numerical measurements.

The Koopman decomposition theory assumes that an infinitesimal operator ${{\mathcal K}^t}$ exists that maps every vector column onto the next one:
\begin{equation}
\left\{ {{u_0},\;{u_1} = {{\mathcal K}^t}{u_0},\;{u_2} = {{\mathcal K}^t}{u_1} = {{\left( {{{\mathcal K}^t}} \right)}^2}{u_0},.\;..,\;{u_{{N_t}}} = {{\mathcal K}^t}{u_{{N_t} - 1}} = {{\left( {{{\mathcal K}^t}} \right)}^{{N_t}}}{u_0}} \right\}.
\end{equation}

Our aim is to build the best numerical approximation of the Koopman operator using the DMD technique. The next step consists in forming two data matrices from the observables sequence, in the form:
\begin{equation}
{V_0} = \left[ {\begin{array}{*{20}{c}}
{{u_0}}&{{u_1}}&{...}&{{u_{{N_t} - 1}}}
\end{array}} \right] \in {\mathbb{R}^{{N_x} \times {N_t}}},\;{V_1} = \left[ {\begin{array}{*{20}{c}}
{{u_1}}&{{u_2}}&{...}&{{u_{{N_t}}}}
\end{array}} \right] \in {\mathbb{R}^{{N_x} \times {N_t}}}.
\end{equation}

Assume that over a sufficiently long sequence of snapshots, the latest snapshot may be expressed as a linear combination of preceding vectors, so that:
\begin{equation}
{u_{{N_t}}} = {c_0}{u_0} + {c_1}{u_1} + ... + {c_{{N_t} - 1}}{u_{{N_t} - 1}} + {\mathcal R},
\end{equation}
where ${c_i} \in \mathbb{R},i = 0,...,N - 1$ and ${\mathcal R}$ is the residual vector. The following relations are true:
\begin{equation}\label{koopth1}
\left\{ {{u_1},{u_2},...{u_{{N_t}}}} \right\} = {{\mathcal K}^t}\left\{ {{u_0},{u_1},...{u_{{N_t} - 1}}} \right\} = \left\{ {{u_1},{u_2},...,{V_0}c} \right\} + {\mathcal R},
\end{equation}
where $c = {\left( {\begin{array}{*{20}{c}}
{{c_0}}&{{c_1}}&{...}&{{c_{N - 1}}}
\end{array}} \right)^T}$ is the unknown column vector. Eq.(\ref{koopth1}) is equivalent to the following relation:
\begin{equation}\label{koopth2}
{{\mathcal K}^t}{V_0} = {V_0}{\mathcal S} + {\mathcal R},\quad {\mathcal S} = \left( {\begin{array}{*{20}{c}}
0&{...}&0&{{c_0}}\\
1&{}&0&{{c_1}}\\
 \vdots & \vdots & \vdots & \vdots \\
0& \ldots &1&{{c_{{N_t} - 1}}}
\end{array}} \right),
\end{equation}
where ${\mathcal S}$ is the companion matrix.

The relationship (\ref{koopth2}) is true when the residual is minimized. It follows that the vector $c$ must be chosen such that ${\mathcal R}$ is orthogonal to $span\left\{ {{u_0},...,{u_{{N_t} - 1}}} \right\}$.
The goal of dynamic mode decomposition is to solve the eigenvalue problem of the companion matrix:
\begin{equation}\label{eigenprob}
{V_1} = {{\mathcal K}^t}{V_0} = {V_0}{\mathcal S} + {\mathcal R},
\end{equation}
where ${\mathcal S}$ approximates the eigenvalues of the Koopman operator ${{\mathcal K}^t}$ when ${\left\| {\mathcal R} \right\|_2} \to 0$.

As a direct result of resolving the minimization problem (\ref{eigenprob}), minimizing the residual enhances overall convergence, and so the eigenvalues and eigenvectors of ${\mathcal S}$ will converge toward the eigenvalues and eigenvectors of the Koopman operator, respectively. 

The advantage of this method is that the Koopman operator has an infinite number of eigenvalues, whereas its DMD approximation is linear and has a finite number of terms. 

Model reduction is highly dependent on the selection of dynamic modes. The superposition of all Koopman modes, weighted by their amplitudes and complex frequencies, approximates the whole data sequence, but some modes contribute insignificantly. In this research, we create a reduced-order model of the data that only includes the most important modes that make a substantial contribution to the representation of the data, which we refer to as the leading modes.

The data snapshots at every time step will be represented as a Koopman spectral decomposition of the form:
\begin{equation}\label{koopth3}
{u_{DMD}}\left( {{\bf{x}},{t_i}} \right) = \sum\limits_{j = 1}^{{N_{DMD}}} {{a_j}\left( {{t_i}} \right)\lambda _j^{i - 1}{\phi _j}\left( {\bf{x}} \right)} ,\quad {\mkern 1mu} i \in \left\{ {1,...,{N_t}} \right\},\quad {\mkern 1mu} {t_i} \in \left\{ {{{\rm{t}}_1},...,{{\rm{t}}_{{N_t}}}} \right\},
\end{equation}
where ${N_{DMD}} \ll {N_t}$ represents the number of Koopman leading modes ${\phi}\left( {\bf{x}} \right) $ involved in the spectral decomposition of data snapshots, ${\lambda _j}$ are the Koopman eigenvalues, and ${a_j} \in \mathbb{C}$ are the modal amplitudes of the Koopman modes, respectively.

 The leading modes indicate a subset of Koopman modes that will be chosen from all computed DMD modes using an original criterion, discussed in the following.

We define the weight of each Koopman mode as follows:
\begin{equation}\label{entropy}
{w\mathcal{K}_j} = \int\limits_{\Delta t}^{{t_{{N_t}}}} {\sum\limits_{i = 1}^{{N_t}} {{a_j}\left( t \right)\lambda _j^{i - 1}} } dt,
\end{equation}
where ${\lambda _j}$ are the Koopman eigenvalues, and ${a_j} \in \mathbb{C}$ are the modal amplitudes of the Koopman modes, respectively.

Let
\begin{equation}\label{error}
E{r_{DMD}} = \frac{{{{\left\| {u\left( {\bf{x}} \right) - {u_{DMD}}\left( {\bf{x}} \right)} \right\|}_2}}}{{{{\left\| {u\left( {\bf{x}} \right)} \right\|}_2}}},
\end{equation}
be the relative error of the difference between the variables of the full model and approximate DMD solutions over the exact one, where $u\left( {\bf{x}} \right)$ represents the full solution of the model and ${u_{DMD}}\left( {\bf{x}} \right)$ represents the reduced order solution. 

The leading dynamic modes and their related frequencies are chosen in descending order of the modal entropy, until a minimal relative error of the reduced-order model is obtained. To produce the reduced-order model amounts to finding the solution to the following minimisation problem:
\begin{equation}\label{minprob}
\left\{ \begin{array}{l}
Find\;{N_{DMD}} \in N,\;w.r.t.\quad {u_{DMD}}\left( {{\bf{x}},{t_i}} \right) = \sum\limits_{j = 1}^{{N_{DMD}}} {{a_j}{\phi _j}\left( {\bf{x}} \right)\lambda _j^{i - 1}} ,\\
\quad \quad i \in \left\{ {1,...,{N_t}} \right\},\quad {t_i} \in \left\{ {{{\rm{t}}_1},...,{{\rm{t}}_{{N_t}}}} \right\},\\
Subject\;to\quad \mathop {\arg \min }\limits_{{N_{DMD}}} \left\{ {{w\mathcal{K}_1} > {w\mathcal{K}_2} > ... > {w\mathcal{K}_{{N_{DMD}}}},\;E{r_{DMD}} \le \varepsilon } \right\}.
\end{array} \right.
\end{equation}

As a consequence, the modes and frequencies with the highest effect on approximation accuracy are selected to be included in the model with a reduced computational complexity.

\section{Reduced-order modelling of Saint-Venant equations model}

The test problem used in this paper consists of the nonlinear Saint-Venant equations (also called the shallow water equations \cite{SaintVenant}) in a channel on the rotating earth:
\begin{equation}\label{sw1}
\frac{{\partial \left( {\tilde u\widetilde h} \right)}}{{\partial t}} + \frac{{\partial \left( {{{\tilde u}^2}\widetilde h + g{{\widetilde h}^2}/2} \right)}}{{\partial x}} + \frac{{\partial \left( {\tilde u\tilde v\widetilde h} \right)}}{{\partial y}} = \widetilde h\left( {f\tilde v - g\frac{{\partial H}}{{\partial x}}} \right),
\end{equation}
\begin{equation}\label{sw2}
\frac{{\partial \left( {\tilde v\widetilde h} \right)}}{{\partial t}} + \frac{{\partial \left( {\tilde u\tilde v\widetilde h} \right)}}{{\partial x}} + \frac{{\partial \left( {{{\tilde v}^2}\widetilde h + g{{\widetilde h}^2}/2} \right)}}{{\partial y}} = \widetilde h\left( { - f\tilde u - g\frac{{\partial H}}{{\partial y}}} \right),
\end{equation}
\begin{equation}\label{sw3}
\frac{{\partial \widetilde h}}{{\partial t}} + \frac{{\partial \left( {\tilde u\widetilde h} \right)}}{{\partial x}} + \frac{{\partial \left( {\tilde v\widetilde h} \right)}}{{\partial y}} = 0,
\end{equation}
where $\tilde u$ and $\tilde v$ are the velocity components in the $\tilde x$ and $\tilde y$ axis directions respectively, $\tilde h$ represents the
depth of the fluid, $H\left( {x,y} \right)$ is the the orography field, $\tilde f$ is the Coriolis factor and $g$ is the acceleration of gravity.

The reference computational configuration is the rectangular $2D$ domain $\Omega = \left[ {0,{L_{\max }}} \right] \times \left[{0,{D_{\max }}} \right]$. Subscripts represent the derivatives with respect to time and the streamwise and spanwise coordinates.

The Coriolis parameter is modelled as varying linearly in the spanwise direction, such that
\begin{equation}\label{coriolis}
\widetilde f = {f_0} + \beta (\widetilde y - {D_{\max }}),
\end{equation}
where ${f_0},\beta $ are constants, ${L_{\max }},{D_{\max }}$ are the dimensions of the rectangular domain of integration.

The height of the orography is given by the fixed two-dimensional field
\begin{equation}\label{oro}
H\left( {x,y} \right) = \alpha {e^{{y^2} - {x^2}}}.
\end{equation}

The model (\ref{sw1})-(\ref{sw3}) is associated with periodic boundary conditions in the  $\tilde x$-direction and solid wall boundary condition in
the $\tilde y$-direction:
\begin{equation}\label{cond1}
\tilde u\left( {0,\tilde y,\tilde t} \right) = \tilde u\left( {{L_{\max }},\tilde y,\tilde t} \right),\;\tilde v\left( {\tilde x,0,\tilde t} \right) = \tilde v\left( {\tilde x,{D_{\max }},\tilde t} \right) = 0,
\end{equation}
and also with the initial Grammeltvedt type condition \cite{Grammeltvedt1969} as the initial height field, which propagates the energy in wave number one, in the
streamwise direction:
\begin{equation}\label{cond2}
{h_0}\left( {\tilde x,\tilde y} \right) = {H_0} + {H_1}\tanh \left( {\frac{{10({D_{\max }}/2 - \tilde y)}}{{{D_{\max }}}}} \right) + {H_2}\sin \left( {\frac{{2\pi \tilde x}}{{{L_{\max }}}}} \right){\cosh ^{ - 2}}\left( {\frac{{20({D_{\max }}/2 - \tilde y)}}{{{D_{\max }}}}} \right).
\end{equation}

Using the geostrophic relationship $\tilde u =  - {\tilde h_{\tilde y}}\left( {g/\tilde f} \right)$, $\tilde v = {\tilde h_{\tilde x}}\left(
{g/\tilde f} \right)$, the initial velocity fields are derived as:
\[{u_0}\left( {\tilde x,\tilde y} \right) =  - \frac{g}{{\tilde f}}\frac{{10{H_1}}}{{{D_{\max }}}}\left( {{{\tanh }^2}\left( {\frac{{5{D_{\max }} - 10\tilde y}}{{{D_{\max }}}}} \right) - 1} \right) - \]
\begin{equation}\label{cond3}
\frac{{18g}}{{\tilde f}}{H_2}\sinh \left( {\frac{{10{D_{\max }} - 20\tilde y}}{{{D_{\max }}}}} \right)\frac{{\sin \left( {\frac{{2\pi \tilde x}}{{{L_{\max }}}}} \right)}}{{{D_{\max }}{{\cosh }^3}\left( {\frac{{10{D_{\max }} - 20\tilde y}}{{{D_{\max }}}}} \right)}},
\end{equation}
\begin{equation}\label{cond4}
{v_0}\left( {\tilde x,\tilde y} \right) = 2\pi {H_2}\frac{g}{{\tilde f{L_{\max }}}}\cos {\mkern 1mu} \left( {\frac{{2\pi \tilde x}}{{{L_{\max }}}}} \right){\cosh ^{ - 2}}\left( {\frac{{20({D_{\max }}/2 - \tilde y)}}{{{D_{\max }}}}} \right).
\end{equation}

The constants used for the test problem are
\[{f_0} = {10^{ - 4}}{s^{ - 1}}, \quad  \alpha  = 4000, \quad \beta  = 1.5 \times {10^{ - 11}}{s^{ - 1}}{m^{ - 1}},\quad g = 9.81m{s^{ - 1}},\]
\[ {{\rm{D}}_{\max }}{\rm{ = 60}} \times {\rm{1}}{{\rm{0}}^3}{\rm{m,}}\quad {{\rm{L}}_{\max }}{\rm{ = 265}} \times {\rm{1}}{{\rm{0}}^3}{\rm{m}},  \]
\[{H_0} = 10 \times {\rm{1}}{{\rm{0}}^3}m,\quad {H_1} = -700m,\quad {H_2} = -400m. \]

The error of the numerical algorithm is set to be less than $\varepsilon=10^{-7}$. A non-dimensional analysis was performed to assess the performances of the reduced-order shallow water model. 
Reference quantities of the dependent and independent variables in the shallow water model are considered, i.e. the length scale
${L_{ref}} = {L_{\max }}$ and the reference units for the height and velocities, respectively, are given by the initial conditions ${h_{ref}} =
{h_0}$, ${u_{ref}} = {u_0}$. A typical time scale is also considered, assuming the form ${t_{ref}} = {L_{ref}}/{u_{ref}}$. 

In order to make the
system of equations (\ref{sw1})-(\ref{sw3}) non-dimensional, the non-dimensional variables
\begin{equation}\label{nondim}
\left( {t,x,y} \right) = \left( {\tilde
t/{t_{ref}},\tilde x/{L_{ref}},\tilde y/{L_{ref}}} \right),\quad \left( {h,u,v} \right) = \left( {\tilde h/{h_{ref}},\tilde u/{u_{ref}},\tilde
v/{u_{ref}}} \right)
\end{equation}
are introduced. 

The numerical results are obtained employing a Lax-Wendroff finite difference discretization scheme \cite{Brass2011} and used in further numerical experiments in dimensionless form. The training data
comprises a number of $289$ unsteady solutions of the two-dimensional shallow water equations model (\ref{sw1})-(\ref{sw3}), at regularly spaced time
intervals of $\Delta t = 1800s$ for each solution variable.

The numerical results of two tests illustrating the computing performance of the approach are presented below. In the first experiment, the threshold is set to be $\varepsilon  = {10^{ - 3}}$ for solving the optimization problem (\ref{minprob}). In the second experiment, the threshold is set at $\varepsilon  = {10^{ - 4}}$ for solving the optimization problem (\ref{minprob}).

Figures \ref{fig1}--\ref{fig3} present the spectrum of Koopman
decomposition eigenvalues, of geopotential height field $h$, streamwise field $u$ and spanwise field $v$, respectively, in the case of two experiments, and the leading Koopman modes selected by resolving the optimization problem (\ref{minprob}). In the second experiment, extra modes are selected (darker colored dots) to improve the reduced-order model precision.
\begin{figure}[h!]
\centering
a.\includegraphics[scale=0.4]{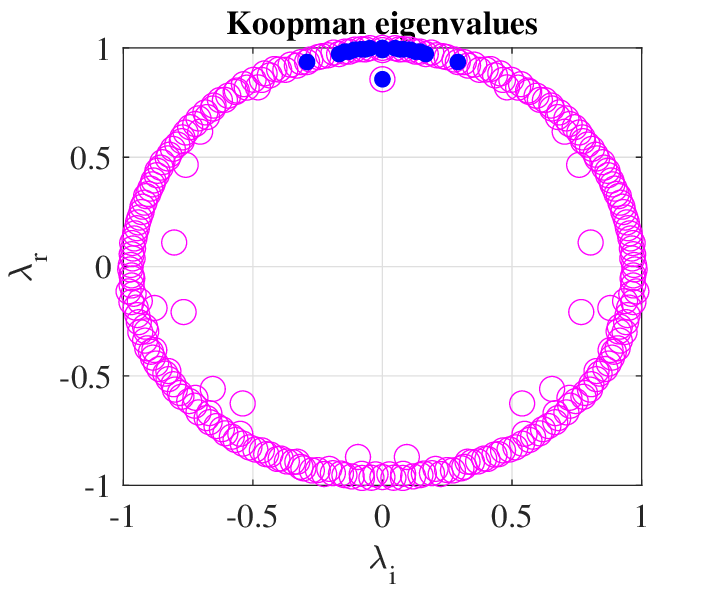}
b.\includegraphics[scale=0.4]{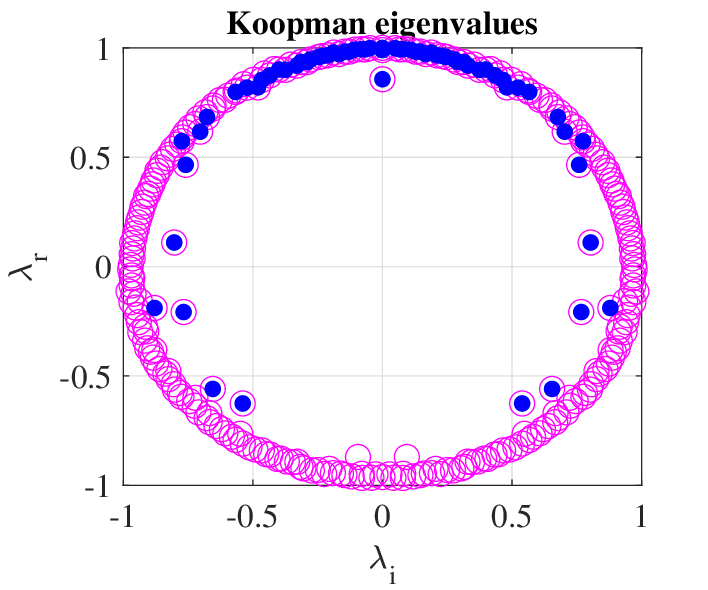}
\caption{The spectrum of Koopman decomposition of height field $h$: a) in the first experiment $\left( {\varepsilon  = {{10}^{ - 3}}} \right)$, $21$ leading modes are selected (darker colored dots); b) in the second experiment $\left( {\varepsilon  = {{10}^{ - 4}}} \right)$, $67$ leading modes are selected (darker colored dots) \label{fig1} }
\end{figure}
\begin{figure}[h!]
\centering
a.\includegraphics[scale=0.4]{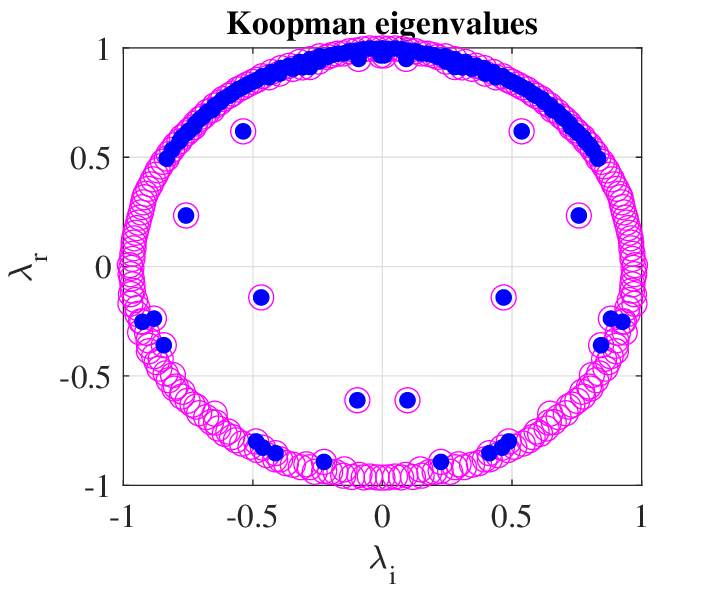}
b.\includegraphics[scale=0.4]{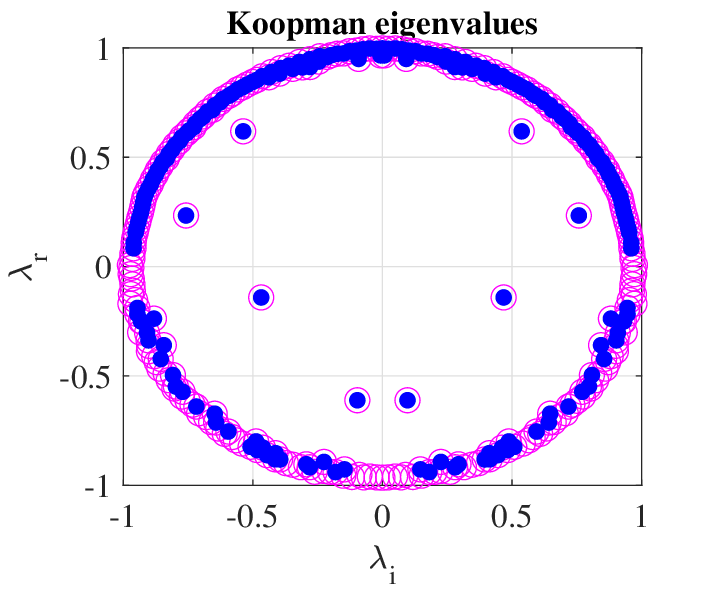}
\caption{The spectrum of Koopman decomposition of streamwise field $u$: a) in the first experiment $\left( {\varepsilon  = {{10}^{ - 3}}} \right)$, $116$ leading modes are selected (darker colored dots); b) in the second experiment $\left( {\varepsilon  = {{10}^{ - 4}}} \right)$, $199$ leading modes are selected (darker colored dots) \label{fig2} }
\end{figure}

\begin{figure}[h!]
\centering
a.\includegraphics[scale=0.4]{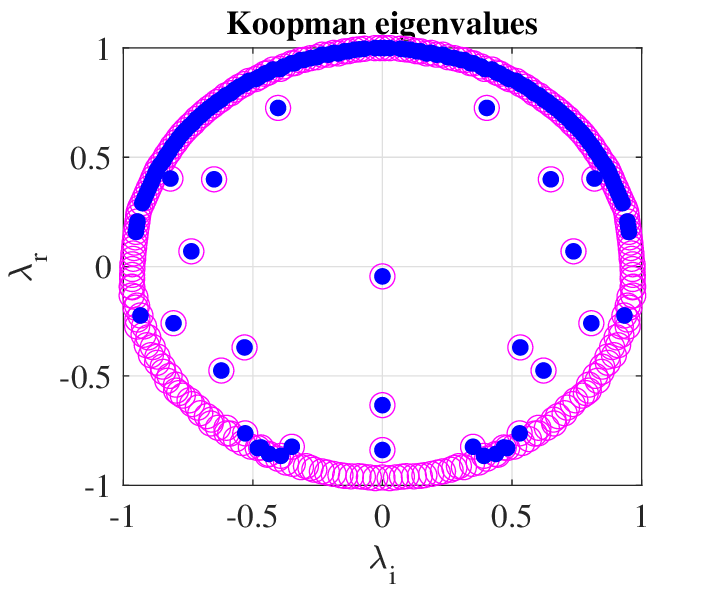}
b.\includegraphics[scale=0.4]{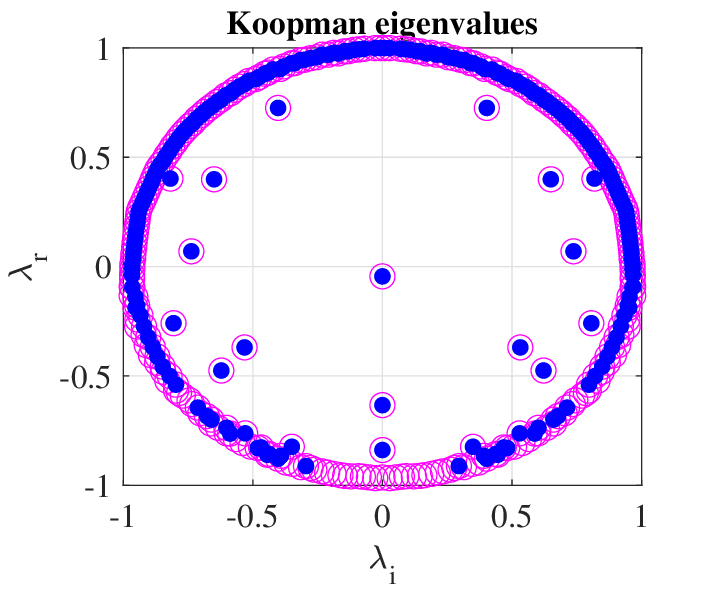}
\caption{The spectrum of Koopman decomposition of spanwise field $v$: a) in the first experiment $\left( {\varepsilon  = {{10}^{ - 3}}} \right)$, $151$ leading modes are selected (darker colored dots); b) in the second experiment $\left( {\varepsilon  = {{10}^{ - 4}}} \right)$, $212$ leading modes are selected (darker colored dots) \label{fig3} }
\end{figure}

The representation of the height field compared to its reduced-order model is displayed in Figures \ref{fig4}--\ref{fig5}, in the case of both experiments.
\begin{figure}[h!]
\centering
\includegraphics[scale=0.4]{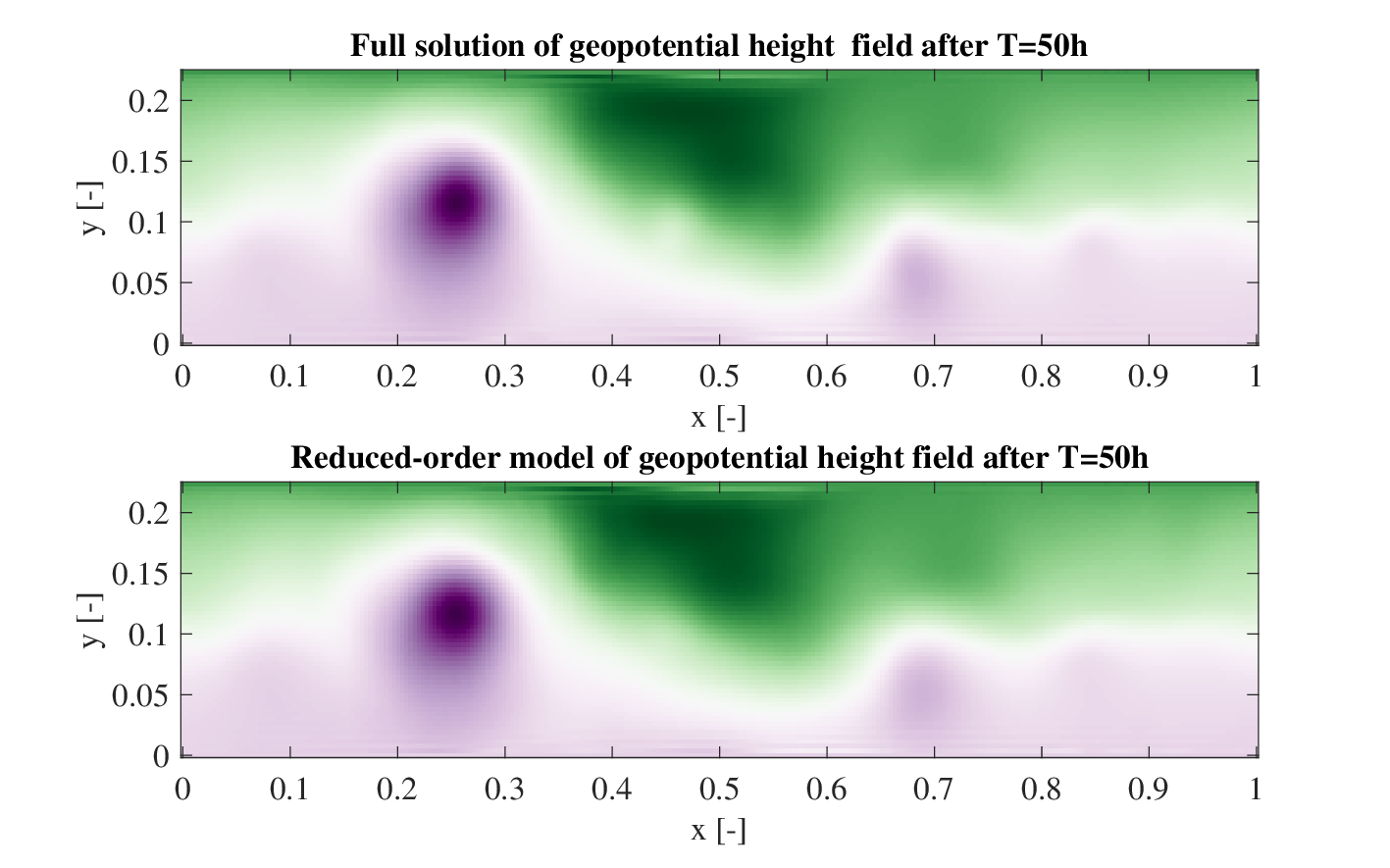}
\caption{Full solution of height field $u$ after $T=50h$, compared to its reduced-order model, in the case of the first experiment, the relative error is of order $\mathcal{O}\left( 10^{-3} \right)$\label{fig4} }
\end{figure}
\begin{figure}[h!]
\centering
\includegraphics[scale=0.4]{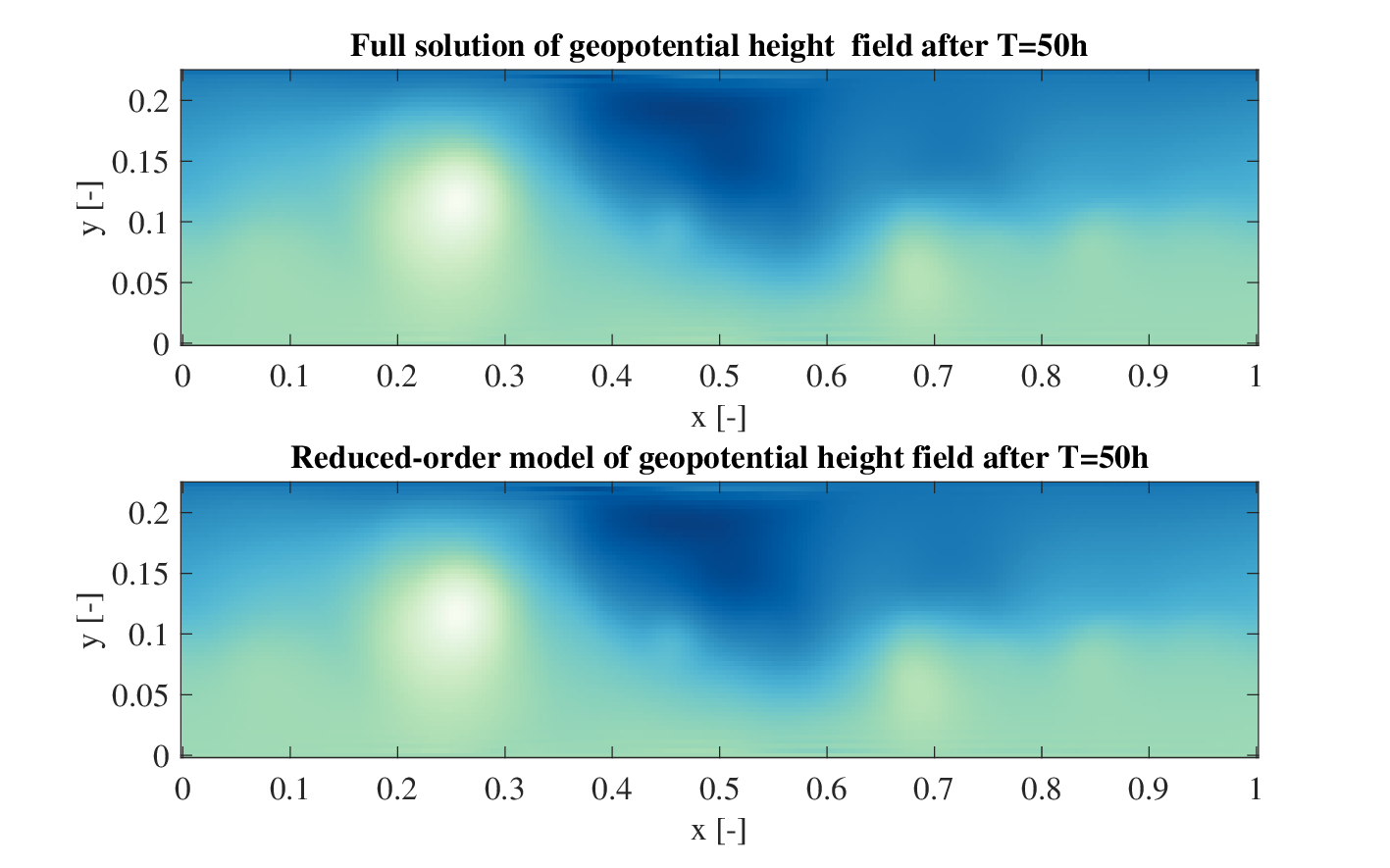}
\caption{Full solution of height field $u$ after $T=50h$, compared to its reduced-order model, in the case of the second experiment, the relative error is of order $\mathcal{O}\left( 10^{-4} \right)$\label{fig5} }
\end{figure}

The vorticity field compared to its reduced-order model is illustrated in Figures \ref{fig6}--\ref{fig7}, in the case of both experiments, at different time instances.
\begin{figure}[h!]
\centering
\includegraphics[scale=0.4]{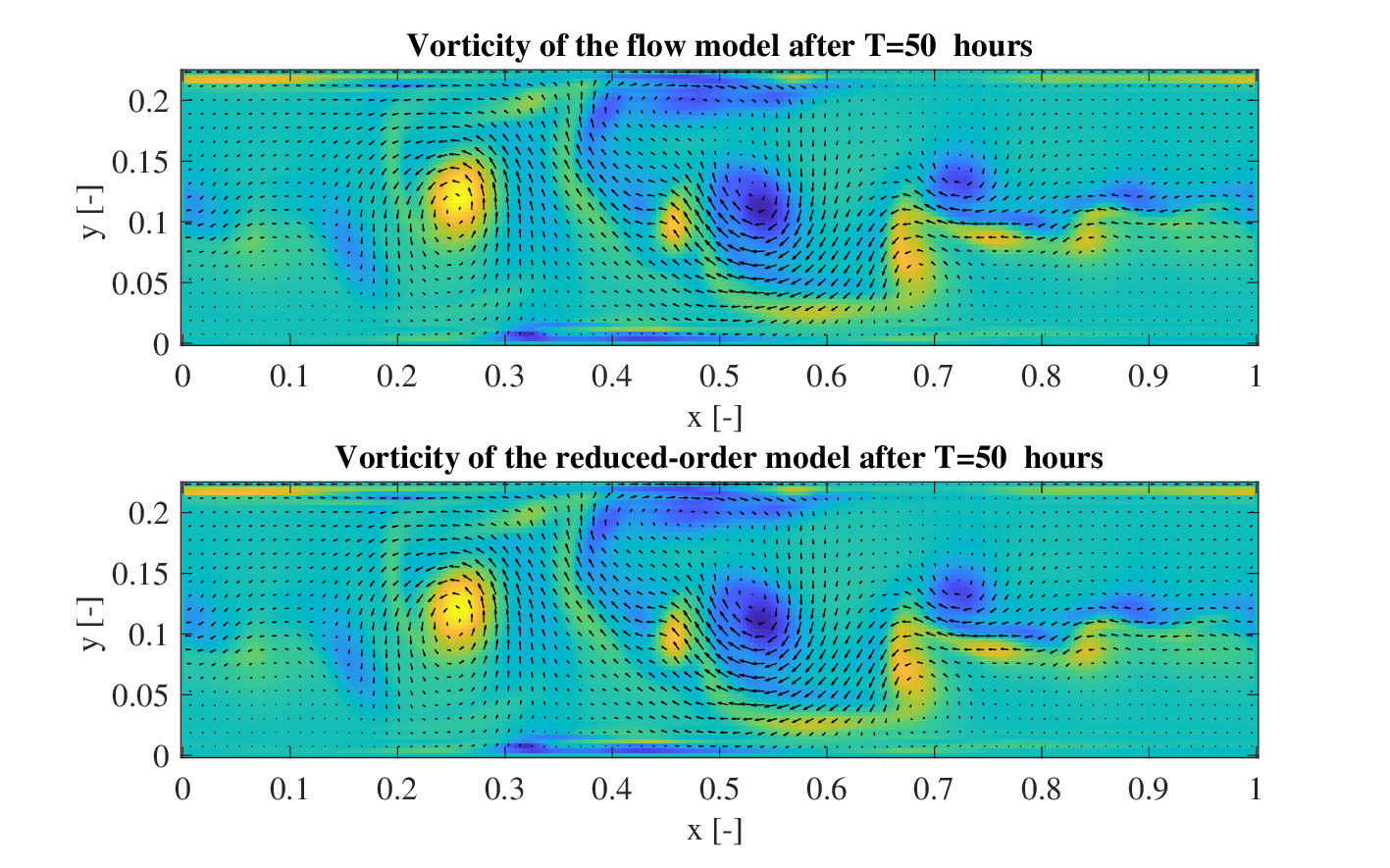}
\caption{Vorticity field after $T=50h$, compared to its reduced-order model, in the case of the first experiment, the relative error is of order $\mathcal{O}\left( 10^{-3} \right)$\label{fig6} }
\end{figure}
\begin{figure}[h!]
\centering
\includegraphics[scale=0.4]{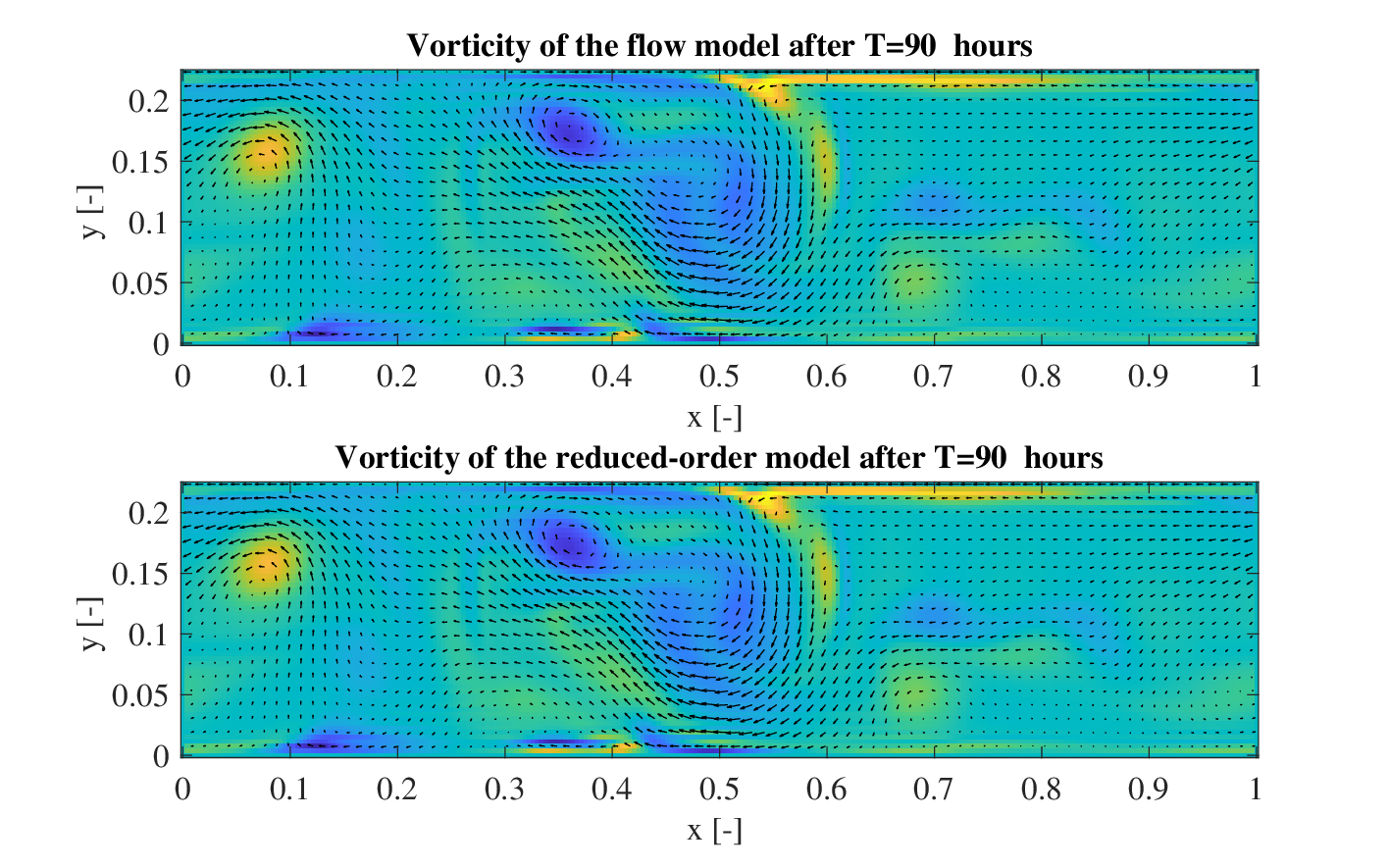}
\caption{Vorticity field after $T=90h$, compared to its reduced-order model, in the case of the second experiment, the relative error is of order $\mathcal{O}\left( 10^{-4} \right)$\label{fig7} }
\end{figure}

Table \ref{table1} presents the percentage reduction of the computational complexity of the reduced-order model, in the two experiments performed.

\begin{center}
\begin{table}[h!]
\caption{The percentage reduction of the computational complexity of the reduced-order model\label{table1}}
\noindent \fontsize{7}{8}\selectfont
%\centering
\begin{tabular}{ | c | c | c |  c | }
\hline
\multirow{1}{*}{Full model } & \multirow{1}{*}{Full model } & \multirow{1}{*}{First test:} & \multirow{1}{*}{Second test:} \\
\multirow{1}{*}{components} & \multirow{1}{*}{ rank} & \multirow{1}{*}{Reduced-order rank,} & \multirow{1}{*}{Reduced-order rank,} \\
\multirow{1}{*}{ } & \multirow{1}{*}{ } & \multirow{1}{*}{Percentage reduction} & \multirow{1}{*}{Percentage reduction} \\
\hline
Height field $h$ & $288$ & $21,\;92.70\% $ & $67,\;76.73\% $\\
\hline
Streamwise field $u$ & $288$ & $116,\;59.72\% $ & $199,\;30.90\% $\\
\hline
Spanwise field $v$ & $288$ & $151,\;47.56\% $ & $212,\;26.38\% $\\
\hline
\end{tabular}
\end{table}
\end{center}
%

%%%%%%%%%%%%%%%%%%%%%%%%%%
\section{Conclusions}

The current study concentrated on a topic of significant interest in fluid dynamics: the identification of a model of reduced computational complexity from numerical code output, based on Koopman operator Theory.
The full model consisted in the Saint-Venent equations model, that have been computed using a Lax-Wendroff finite difference discretization scheme. The Koopman composition operator have been numerically approximated with the algorithm of reduced-order modelling, endowed with a novel criterion of selection of the most influential Koopman modes, based on the modes weights. It automatically selects the most representative Koopman modes, even if they exhibit rapid development with lower amplitudes or are composed of high amplitude fast damped modes.

Two tests were carried out in order to evaluate the algorithm's computing efficiency in order to enhance the reduced-order model precision. It was demonstrated that the model rank may be decreased by up to $92\%$ without compromising model accuracy.

This approach is a useful tool for creating reduced-order models of complex flow fields characterized by non-linear models.

%%%%%%%%%%%%%%

\end{document}